\documentclass{article}
\usepackage{amsmath, amsthm, amssymb, amsfonts}
\usepackage{mathtools}
\usepackage{graphics}
\usepackage{epstopdf}
\usepackage[longnamesfirst, comma]{natbib}

\newcommand{\E}{{\mathbb E}}

\newcommand{\R}{{\mathbb R}}

\newcommand{\CC}{{\mathcal{C}}}

\newcommand{\FF}{{\mathcal{F}}}

\newcommand{\bsx}{\boldsymbol x}
\newcommand{\bsy}{\boldsymbol y}

\newcommand{\bsU}{\boldsymbol U}

\newcommand{\bsX}{\boldsymbol X}
\newcommand{\bsY}{\boldsymbol Y}
\newcommand{\bsZ}{\boldsymbol Z}
\newcommand{\bsone}{\boldsymbol 1}

\newcommand{\bszero}{\boldsymbol 0}

\newcommand{\bsmu}{\boldsymbol \mu}

\newtheorem{theorem}{Theorem}

\newtheorem{corollary}[theorem]{Corollary}

\newtheorem{lemma}[theorem]{Lemma}

\newtheorem{proposition}[theorem]{Proposition}

\theoremstyle{definition}

\newtheorem{example}[theorem]{Example}
\newtheorem{definition}[theorem]{Definition}
\newtheorem{remark}[theorem]{Remark}

\begin{document}

\title{Dependence uncertainty bounds for the energy score and the multivariate Gini mean difference}
\author{ Carole Bernard \\
%EndAName
\texttt{carole.bernard@grenoble-em.com}\\
Department of Accounting, Law and Finance\\
Grenoble Ecole de Management,\\
Grenoble, France. 
\and 
Alfred M\"{u}ller\\
%EndAName
\texttt{mueller@mathematik.uni-siegen.de}\\
Department of Mathematics,\\
University of Siegen,\\
Siegen, Germany. }
\date{\today }
\maketitle

\begin{abstract} The energy distance and energy scores became important tools in multivariate statistics and multivariate probabilistic forecasting in recent years. They are both based on the expected distance
of two independent samples. In this paper we study dependence uncertainty bounds for these quantities under the assumption that we know the marginals but do not know the dependence structure. We find some interesting sharp analytic bounds, where one of them is obtained for an unusual spherically symmetric copula. These results should help to better understand the sensitivity of these measures to misspecifications in the copula.
\end{abstract}
\textbf{Keywords}: dependence uncertainty bounds,  energy score, Gini mean difference, spherically symmetric copula.

\section{Introduction}

In recent years the so-called {\it energy distance} became a famous tool in multivariate statistics used e.g., for goodness-of-fit tests and many other things. For a good overview over this topic we refer to \cite{szekely2017}. 
Similar concepts have been suggested in the theory of multivariate probabilistic forecasting, where the so-called {\it energy score} has been suggested as a strictly proper scoring rule for multivariate distributions
in the fundamental paper of \cite{gneiting2007}. Both concepts rely on functionals that are based on expected distances of independent copies of random vectors. This is related to the multivariate Gini mean difference, which has been studied in detail in \cite{koshevoy1997}.  In the univariate case the Gini mean difference is a well-known measure of
spread of distributions or inequality in case of income distributions, see e.g., \cite{yitzhaki2003} for an overview. 

In goodness-of-fit testing as well as in probabilistic forecasting one is interested in detecting misspecifications of stochastic models. Therefore it is an important question how sensitive the used functionals react to which kind of misspecification. \cite{pinson2013} studied the discrimination ability of the energy score for the case of multivariate normal distributions. Based on simulation studies they conclude that the discrimination ability of the energy score may be limited when focusing on the dependence structure of multivariate probabilistic forecasts, but to the best of our knowledge there has been no general study of this problem so far for general distributions. 

In this paper we want to study this problem of so-called {\it dependence uncertainty bounds} for such quantities like the energy score and the Gini mean difference. By dependence uncertainty bounds we mean here bounds for a functional of a multivariate distribution under the assumption that we only know the marginal distributions but do not know the dependence structure, i.e., we do not know the copula. The study of such uncertainty bounds has a long history going back to \cite{hoffding1940} and \cite{frechet1951}. They considered this problem for correlation coefficients and for the value of cumulative distribution functions. In the meanwhile there is a vast literature on this topic for many kinds of functionals. For an overview see \cite{puccetti2015}. Very often the extremal positive dependence is given by the comonotone copula, in particular if the functional is an expectation of a supermodular function, as has been shown in \cite{tchen1980} and \cite{ruschendorf1983}. It is typically more complicated to find the extremal negative dependence, even in the case of expectations of functions and thus linear functionals of the distributions, which is the case for most problems considered in the literature. An example of a non-linear problem is the case of finding the solution of an optimal stopping problem that was considered in \cite{muller2001}. In such a case of a non-linear problem the characterization of the extremal dependence can be very different from the case of a linear problem.

In this paper we also deal with a non-linear problem, but it will turn out that still the comonotone copula will typically lead to the extremal positive dependence.  But for the extremal negative dependence we find in some cases a very interesting solution based on a spherical symmetric copula. This is an interesting copula, which does not seem to be well-known in the dependence modelling community.

The paper is organized as follows. In Section \ref{S1}, we recall the definitions of the various concepts. We also introduce some important notation that will be used throughout the manuscript and present the problem that is considered in this paper. In Section \ref{S2}, we focus on the expected distance between two multivariate distributions and its sensitivity to dependence uncertainty. Finally in Section \ref{S3}, we provide a number of results on the dependence uncertainty bounds on the energy score. Section \ref{S4} concludes with a number of open questions that are left for future research.

\section{Energy score and Gini mean difference\label{S1}}

Let $\bsX, \tilde{\bsX}$ be independent copies of a $d$-dimensional random vector with cumulative distribution function (cdf)
$$
F(\bsx) = P(X_1 \le x_1,\ldots, X_d \le x_d), \quad \bsx = (x_1,\ldots,x_d) \in \R^d,
$$ 
and $\bsY, \tilde{\bsY}$ be independent copies of a random vector with cdf $G$.
We define the expected distance between two independent $d$-dimensional samples of $F$ and $G$ as
$$
S(F,G) = \E\left( \| \bsX - \bsY \|_2 \right) = \int \|\bsx-\bsy \|_2 F(d\bsx) G(d\bsy),
$$
where we identify the cdfs $F$ and $G$ with the corresponding probability measures and denote as usual by 
$$
\|x \|_2 = \sqrt{\sum_{i=1}^d x_i^2}
$$
the Euclidian distance. For an observation $\bsy,$ we similarly define by identifying $\bsy$ with the one-point measure in $\bsy$
$$
S(F,\bsy) = \E\left( \| \bsX -\bsy \|_2 \right) = \int \|\bsx-\bsy \|_2 F(d\bsx).
$$

The \textit{energy distance} between two distributions $F$ and $G$ is defined as
$$
\mathcal{E}(F,G) = 2 S(F,G) - S(F,F) - S(G,G).
$$
This is a distance between probability distributions, as it can be shown that $\mathcal{E}(F,G) \ge 0$ for all $F,G$ and that $\mathcal{E}(F,G) = 0$
if and only if $F=G$. For details of this concept and applications we refer to the overview article of \cite{szekely2017}. 
A strongly related concept is the so-called \textit{energy score} for a distributional forecast $F$ and an observation $\bsy,$ which is given by 
$$
ES(F,\bsy) = S(F,\bsy) - \frac{1}{2} S(F,F).
$$
This can be generalized by introducing a parameter $\beta\in(0,2)$ as already considered in the fundamental paper of \cite{gneiting2007}.

\begin{definition} For $\beta \in (0,2)$, the generalized expected distance between two independent $d$-dimensional samples of $F$ and $G$ is defined as
$$
S_\beta(F,G) = \E\left( \| \bsX - \bsY \|^\beta_2 \right) = \int \|\bsx-\bsy \|^\beta_2 F(d\bsx) G(d\bsy),
$$ and the \textit{generalized energy score} as
$$
ES_\beta(F,\bsy) = S_\beta(F,\bsy) - \frac{1}{2} S_\beta(F,F) %=  \E\left( \| \bsX - \bsY \|_2^\beta \right) - \frac{1}{2} \E\left( \| \bsX -\bsy \|_2^\beta \right).
$$
Similarly, the corresponding \textit{generalized energy distance} is defined as
\begin{equation}\label{GED}
\mathcal{E}_\beta(F,G) = 2 S_\beta(F,G) - S_\beta(F,F) - S_\beta(G,G).
\end{equation}
\end{definition}
Note that the limiting case $\beta=2$ is excluded in the definition, as $\mathcal{E}_2(F,G)$ only depends on the marginal distributions of $F$ and $G$ and thus does not depend at all
on the copula and in fact therefore is not a distance and does not lead to a proper scoring rule.

\begin{remark}

The function 
$$
S(F,F) =  \E\left( \| \bsX - \tilde{\bsX} \|_2 \right) = \int \|\bsx-\bsy \|_2 F(d\bsx) F(d\bsy)
$$
is known (sometimes up to a constant $2d$) as \textit{multivariate Gini mean difference} and has been studied in detail in \cite{koshevoy1997}. 
To distinguish the univariate version of the Gini mean difference from the multivariate one we denote, from now on, the univariate version with a slight abuse of notation as
$$
M(F) :=  M(F,F) = S(F,F) =  \int |x-y| F(dx) F(dy) = 2 \int (1-F(x)) F(x)  dx
$$ 
and its generalization for $\beta \in (0,2)$ and different $F$ and $G$ similarly as
$$
M_\beta(F,G) = \int |x-y|^\beta F(dx) G(dy).
$$ 
Note that the bordering case of $\beta=2$ yields for $X \sim F$ up to a factor of two the variance
$$
M_2(F,F) =  \int |x-y|^2 F(dx) F(dy) = 2 var(X).
$$ 
\end{remark}

We will frequently consider the random variable $Z = |X-Y|$ for independent random variables $X$ and $Y$. We use the following notation.

\begin{definition}\label{defdia}
For independent $X$ and $Y$ with cdf $F$ and $G$, respectively, we denote by $F \Diamond G$ the cdf of $Z = |X-Y|$ which is given by
$$
F \Diamond G(x) = \int (F(y+x) - F(y-x)) \ G(dy), \quad x \ge 0.
$$
\end{definition}
Notice that the Gini mean difference $M(F) = S(F,F)$ is the mean of $F \Diamond F$ and more general $M_\beta(F,G) = \E(Z^\beta)$ is the corresponding
moment of order $\beta$. 

\begin{example} \label{ex-uniform}
In case of standard uniform distributions $U,V$ on $(0,1)$ for $F=G$ we get for $Z$ 
the density $f_{Z}(z)=  2-2z $ on $[0,1]$ and thus 
\begin{equation}\label{Ex2Mbeta}
M_\beta(F) = \E(Z^\beta) = \int_0^1 z^\beta (2-2z) dz = \frac{2}{\beta+1} - \frac{2}{\beta+2}
\end{equation}
with the special cases $S(F,F) = M(F) = \E Z=\frac{1}{3}$ and in the limiting case  when $\beta = 2$ we get $\E Z^2=\frac{1}{6} = 2 var(U)$. 
\end{example}

\subsection{Dependence Uncertainty Bounds}

We want to study how sensitive these quantities are with respect to the dependence information in the joint distribution. 
Therefore we investigate bounds for such expressions given that we only know the marginals of $F$ and $G$. As usual we denote the marginals by
$$
F_i(x) := P(X_i \le x), \quad x \in \R, \ i = 1,\ldots, d.
$$
By Sklar's theorem we can write the joint cumulative distribution $F$ of $\bsX$ in the form
$$
F(\bsx) = C(F_1(x_1), \ldots, F_d(x_d)), \quad \bsx \in \R^d,
$$
for come copula $C$, see e.g., \cite{nelsen2007}. We denote by $\CC$ the set of all possible copulas and by $\FF = \FF(F_1,\ldots,F_d)$ the so-called Fr\'echet class of all
multivariate distributions with given marginals $F_1,\ldots,F_d$. The well-known Fr\'echet bounds are denoted by 
$$
F^+(\bsx) = \min\{F_1(x_1), \ldots, F_d(x_d)\}
$$
and 
$$
F^-(\bsx) = \max\{F_1(x_1)+ \ldots+ F_d(x_d)-d+1, 0\}
$$
and $C^+$ and $C^-$ will be the corresponding copulas, typically called the comonotonic and countermonotonic copula, where one has to take into account that $C^-$ is only a copula for $d=2$. 

 Similarly, for a joint cumulative distribution function $G$ of $\mathbf{X}$,
$$
G(\bsx) = C(G_1(x_1), \ldots, G_d(x_d)), \quad \bsx \in \R^d,
$$
we denote by $\mathcal{G}$, the Fr\'echet class of all
multivariate distributions with given marginals $G_1,\ldots,G_d$.

We first study the corresponding dependence uncertainty bounds on the generalized expected distance between two independent $d$-dimensional samples from $F$ and $G$. These bounds can then be written as 
\begin{equation}
\inf_{F \in \FF} S_\beta(F,F)\ \hbox{and}\  \sup_{F \in \FF} S_\beta(F,F),\label{supSbeta}
\end{equation}
when both samples come from the same distribution $F$, and by 
\begin{equation}
\inf_{F \in \FF, G\in \mathcal{G}} S_\beta(F,G)\ \hbox{and}\  \sup_{F \in \FF, G\in \mathcal{G}} S_\beta(F,G),\label{supSbeta2}
\end{equation}
when the multivariate distributions $F$ and $G$ are not identical.

For a given observation $\bsy$, we then study dependence uncertainty bounds for its generalized energy score by considering
\begin{equation}\label{infESbeta}
\inf_{F \in \FF} ES_\beta(F,\bsy)\ \hbox{and}\  \sup_{F \in \FF} ES_\beta(F,\bsy)
\end{equation}
 The optimizations in \eqref{supSbeta}, \eqref{supSbeta2} and \eqref{infESbeta} are over the Fr\'echet class $\FF$ and $\mathcal{G}$. In fact, given that the marginal distributions are given,  the uncertainty bounds can also been considered as solutions of optimization problems over the class $\CC$ of all copulas.

\section{Dependence uncertainty bounds for $S_\beta$ \label{S2}}

In this section, we provide  analytic bounds for the expressions \eqref{supSbeta} and \eqref{supSbeta2}. To do so, we first look at some fundamental properties of $S_\beta$ and $ES_{\beta}$.

It is well-known (see for instance \cite{gneiting2007}) that the generalized energy distance $\mathcal{E}_\beta$ is a distance for  $\beta \in (0,2)$, i.e., $\mathcal{E}_\beta(F,G) \ge 0$ for all $F,G$ and therefore also $ES_\beta(F,\bsy) \ge 0$ for all $F$ and $\bsy$. Moreover, this implies by definition \eqref{GED}  also that
\begin{equation} \label{es-sm}
2 S_\beta(F,G) =  S_\beta(F,G) + S_\beta(G,F) \ge  S_\beta(F,F) + S_\beta(G,G).
\end{equation}
It is also well-known that
$ES_\beta$ is a proper scoring rule, meaning that 
\begin{equation} \label{es-proper}
ES_\beta(F,F) \le ES_\beta(F,G) \quad \mbox{ for all } F,G.
\end{equation}
Note also that $ES_\beta(F,F) = \frac{1}{2} S_\beta(F,F)$. We can derive the following lemma on the concavity of $S_\beta(F,F)$.
\begin{lemma}\label{LEM1}
$F \to S_\beta(F,F)$ is concave. 
\end{lemma}
\begin{proof}
Indeed, $S_\beta$ is linear in $F$ and $G$ and therefore we get for $\alpha \in (0,1)$ from \eqref{es-sm} that
\begin{align*}
\lefteqn{S_\beta(\alpha F + (1-\alpha)G, \alpha F + (1-\alpha)G) } \\
& \qquad = \alpha^2 S_\beta(F,F) + \alpha (1-\alpha) (S_\beta(F,G) +  S_\beta(G,F)) + (1-\alpha)^2 S_\beta(G,G)\\
& \qquad \ge \alpha^2 S_\beta(F,F) + \alpha (1-\alpha) ( S_\beta(F,F) +  S_\beta(G,G)) + (1-\alpha)^2 S_\beta(G,G)\\
& \qquad  = \alpha S_\beta(F,F) + (1-\alpha) S_\beta(G,G).
\end{align*}\end{proof}

\subsection{Lower bound on  $S_\beta(F,G)$}
For  finding a minimum value of $S_\beta(F,G)$ the following representation is going to be helpful.
\begin{align} \label{z-representation}
S_\beta( F,G) & =   \E\left( \left\| \bsX - \bsY \right\|_2^\beta \right) =  \ \E \left(\left( \sum_{i=1}^d Z_i^2  \right)^{\frac{\beta}{2}}\right),
 \end{align}
where $Z_i = |X_i-Y_i| \sim F_i \Diamond G_i$ (see Definition \ref{defdia}). 
Thus we have a representation $S_{\beta}(F,G) = \E f(Z_1^2,\ldots,Z_d^2)$ where we 
know the marginals of $(Z_1^2,\ldots,Z_d^2)$ and the function $f$ has the following properties: as a concave function of the sum it is submodular, 
i.e., $-f$ is supermodular. For the definition and properties of supermodular functions and their relevance for inequalities of expectations in case
of distributions with given marginals we refer to Chapter 3 in M\"uller and Stoyan (2002).

From this we can derive the following lower bound.

\begin{theorem} \label{lower-bound-FG}
For any random vector $\bsX$ and $\bsY$ with cumulative cdfs $F$ and $G$  we get the following lower bound:
$$
S_\beta(F,G) \ge \E\left(Z^{\frac{\beta}{2}}\right)
$$
for a random variable $Z$ that is defined as
$$
Z = \sum_{i=1}^d \left((F_i \Diamond G_i)^{-1}(U)\right)^2
$$
for some standard uniform random variable $U$.

In case of identical marginals $F_1 = \ldots = F_d$ and $G_1 = \ldots  =G_d$ this bound is sharp and is obtained for the upper Fr\'echet bounds $F = F^+$ and $G = G^+$. Furthermore, in this case it 
reduces to
$$
S_\beta(F,G) \ge S_\beta\left(F^+,G^+\right) = d^{\frac{\beta}{2}} M_\beta(F_1,G_1).
$$
\end{theorem}

\begin{proof}
According to \eqref{z-representation} we can write
$S(F,G) = \E f(Z_1^2,\ldots,Z_d^2)$ for a submodular function $f$. It follows from \cite{tchen1980} that a lower bound for $\E f(Z_1^2,\ldots,Z_d^2)$ is obtained by assuming that the copula
of $(Z_1^2,\ldots,Z_d^2)$ is given by the upper Fr\'echet bound $C^+$, or equivalently the copula of $(Z_1,\ldots,Z_d)$. This means that we can assume that $Z_i^2 = ((F_i \Diamond G_i)^{-1}(U))^2$ for some fixed 
uniform $U$ and from this the first assertion immediately follows.

If all marginals of $F$ and $G$ are the same, then we have for $F^+$ and $G^+$ that $\bsX = (X_1,\ldots,X_1)$ and $\bsY = (Y_1,\ldots,Y_1)$ and thus we get in 
\eqref{z-representation} also the equality $Z_1 = \ldots = Z_d$ and therefore this lower bound is attained and reduces to 
$$
S_\beta\left(F^+,G^+\right) = \E\left((dZ_1^2)^{\frac{\beta}{2}}\right) = d^{\frac{\beta}{2}} M_\beta(F_1,G_1)
$$
\end{proof}
Note that $S_\beta(F^+,G^+)$ may not be a lower bound when the marginals of $F$ and $G$ are not identical. Consider $X = (4U_1, U_1)$ and  $Y = (U_2, 4U_2)$ for two independent uniform $U_1, U_2$. Then $Z_1$ is large if $U_1$ is large and $Z_2$ is large if $U_2$ is large and they are far away from being comonotone. The support of $(Z_1,Z_2)$ is given in the left panel of Figure \ref{counterex} and does not correspond of the support of the upper Fr\'echet bound. The lower bound is obtained for something different from $F^+$ and $G^+$. In fact, if one takes for $F$ the lower Fr\'echet bound $F= F^-$ and $G = G^+$ instead, i.e., $X = (4U_1, 1-U_1)$ and  $Y = (U_2, 4U_2),$ then we get more positively correlated $Z_1$ and $Z_2$ as depicted in the right panel of Figure \ref{counterex}. Indeed we get $S(F^-,G^+)\approx2.48<S(F^+,G^+)\approx2.55$.

\begin{figure}[!htbp]
\begin{center}
\includegraphics[width=0.4\textwidth,height=\textwidth, angle=270]{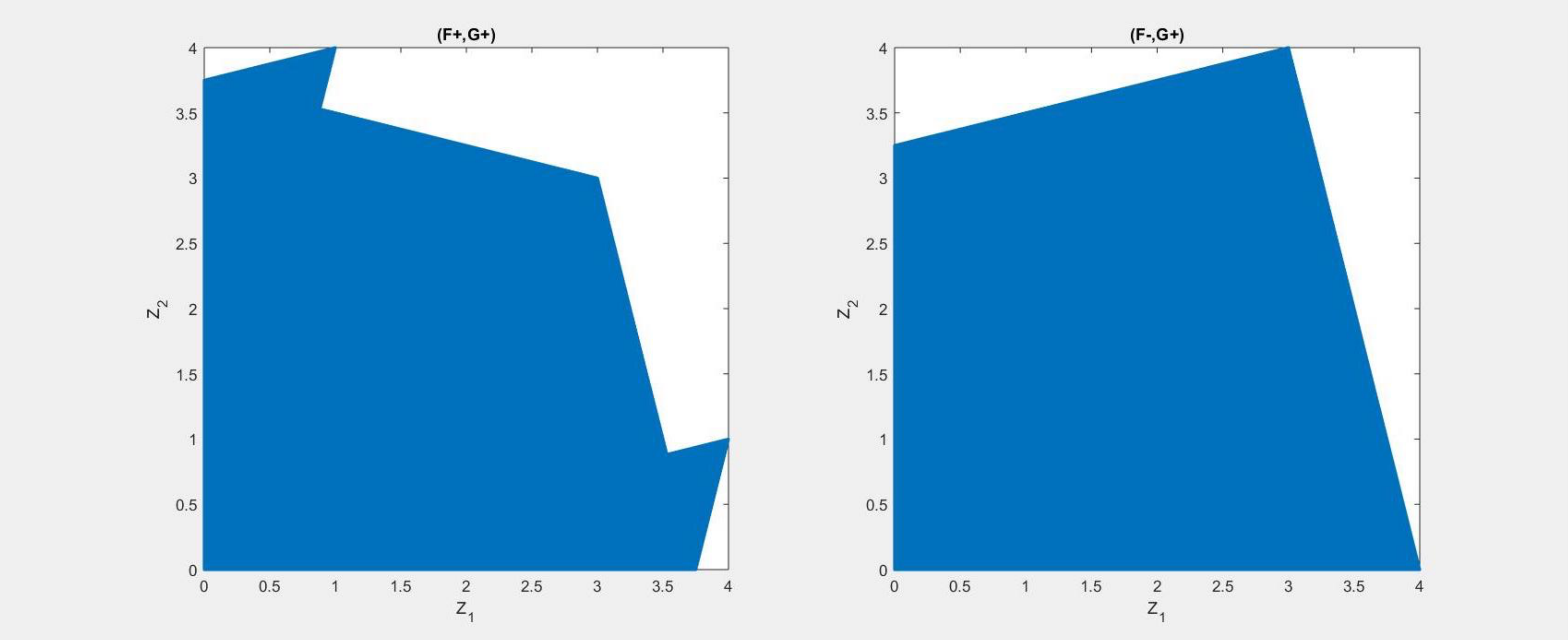}
\end{center}
\caption{Support of $(Z_1,Z_2)$. Left panel: $F=F^+, G=G^+$. Right panel: $F=F^-, G=G^+$.\label{counterex}}
\end{figure}

\subsection{Upper bound on $S_\beta(F,G)$}

It seems to be much more difficult to find a sharp upper bound for $S_\beta(F,G)$, as it is a notoriously difficult problem to find a strongest possible negative dependence in the sense of maximizing the expression
in \eqref{z-representation}. But we can easily derive an upper bound via Jensen's inequality.

\begin{theorem} \label{upper-bound-FG}
For any random vector $\bsX$ and $\bsY$ with cumulative cdfs $F$ and $G$  we get the following upper bound:
$$
S_\beta(F,G) \le \left(  \sum_{i=1}^d M_2(F_i,G_i)\right)^{\frac{\beta}{2}}.
$$
\end{theorem}

\begin{proof}
Applying Jensen's inequality  to \eqref{z-representation} we get that  for any multivariate distributions $F$ and $G$,
$$
S_\beta(F,G) =  \ \E\left( \sum_{i=1}^d Z_i^2  \right)^{\frac{\beta}{2}} \le \  \left(  \sum_{i=1}^d \E Z_i^2  \right)^{\frac{\beta}{2}} \ = \  \left(  \sum_{i=1}^d M_2(F_i,G_i)\right)^{\frac{\beta}{2}} .
$$
\end{proof}

\subsection{Upper and lower bounds on $S_\beta(F,G)$ for copulas}

In the special case of uniform marginal distributions, the class $\FF$ is simply the class of all copulas and we are able to derive explicit expressions of  the lower and upper bounds obtained in Theorems \ref{lower-bound-FG} and \ref{upper-bound-FG}. Specifically, from Example \ref{ex-uniform}, we have expression \eqref{Ex2Mbeta} and we get that $\E Z_i^2 = 1/6$ . We  can then immediately derive the following consequence.

\begin{corollary}\label{upper-bound-S-copula}
For copulas $C_1,C_2$ we get the following  bounds:
$$
S_\beta(C^+,C^+)  = d^{\frac{\beta}{2}} \cdot \left( \frac{2}{\beta+1} - \frac{2}{\beta+2} \right)\le S_\beta(C_1,C_2) \le  \left(  \frac{d}{6}\right)^{\frac{\beta}{2}}
$$
and in particular for $\beta = 1,$
$$
S(C^+,C^+)  = \frac{1}{3}\sqrt{d} \le S(C_1,C_2) \le  \sqrt{\frac{d}{6}}.
$$
\end{corollary}

\begin{example} \label{sc1c2}
One could conjecture that a sharp upper bound for copulas is obtained by $S_\beta(C^-,C^+)$. We will now give an explicit counterexample for the important case of $d = 2$ and $\beta=1$.
Using the invariance under rotation we can easily derive $S(C^-,C^+)$ in this case from the expected distance of two points on the axes. 
$$
S(C^-,C^+) = \frac{1}{\sqrt{2}} \int_0^1 \int_0^1 \sqrt{x^2+y^2} \ dx dy = \frac{\sqrt{2} + \log(1+\sqrt{2})}{3\sqrt{2}} \approx 0.541.
$$
Now let us consider the copula $C^\|$, defined as the distribution of the following random vector $(U_1,U_2)$ with $U_1 \sim U(0,1)$ and 
\begin{equation}
U_2 :=
\left\{
\begin{array}{ll}
U_1 + \frac{1}{2}, & \mbox{ if } U_1 \le \frac{1}{2},\\[3mm]
U_1 - \frac{1}{2}, & \mbox{ if } U_1 > \frac{1}{2}.
\end{array}
\right.
\label{copu}\end{equation}
Thus the support of the copula $C^\|$ consists of two parallel line segments as displayed in Figure \ref{FigSupp}.
\begin{figure}[!htbp]
\begin{center}
\includegraphics[width=0.7\textwidth,height=0.4\textwidth]{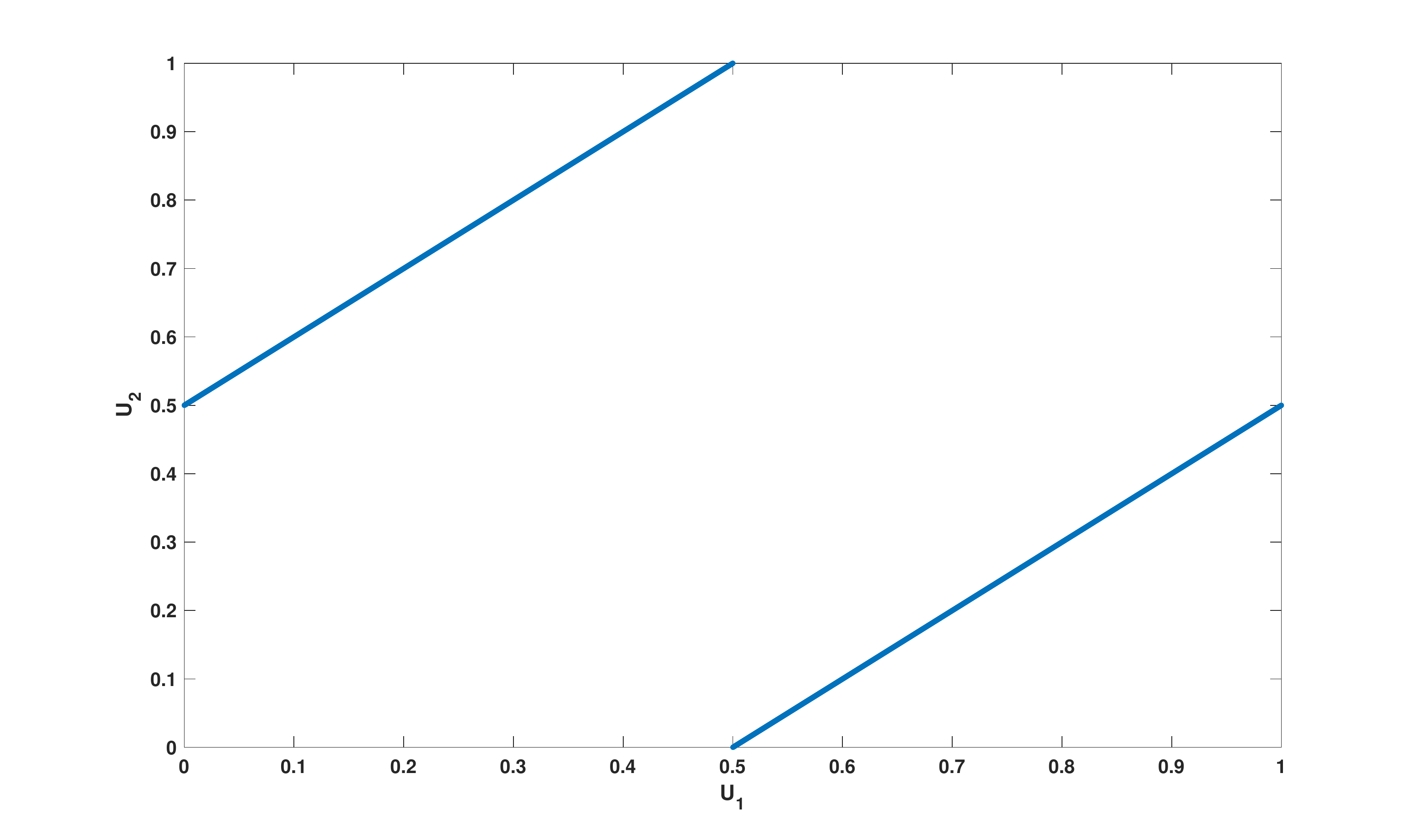}
\end{center}
\caption{Support of the copula $C^\|$ defined in \eqref{copu} \label{FigSupp}}
\end{figure}

Then we get by a similar computation
$$
S(C^+,C^\|) = \frac{1}{\sqrt{2}} \int_0^1 \int_0^1 \sqrt{\frac{1}{2} + \left(\frac{1}{2}+ x-y\right)^2} \ dx dy \approx 0.549.
$$
Thus $S(C^+,C^\|) > S(C^-,C^+)$. 
We notice, however, that this inequality is reversed for the energy distance. We get $\mathcal{E}(C^-,C^+) \approx 0.069 >  0.064 \approx \mathcal{E}(C^-,C^\|)$ even though $S(C^-,C^+) < S(C^+,C^\|)$, as 
$S(C^+,C^+) \approx 0.471$ is significantly smaller than $S(C^\|,C^\|) \approx 0.4985$. Therefore it is still an open problem whether for the energy distance $\mathcal{E}(C^-,C^+)$ maximizes $\mathcal{E}(C_1,C_2)$ among all copulas.
\end{example}

\subsection{Upper bound on $S(F,F)$ for copulas}

Under the assumption of equal copulas $C = C_1 = C_2,$ we can improve the upper bounds in Corollary \ref{upper-bound-S-copula} and even find sharp bounds for $S(C,C)$ in the case
of dimension $d=2$ and $d=3$ for some interesting copulas, which can be called spherical symmetric copulas. These do not seem to be very well-known in the community working on dependence modelling and copulas, but they have been considered from time to time in the statistics literature, e.g., in \cite{eaton1981}, \cite{schwarz1985} and \cite{perlman-wellner-11}.
\cite{perlman-wellner-11} show that in dimensions $d=2$ and $d=3$ there are spherical symmetric random vectors $\bsX$ whose marginals are uniformly distributed on $[-1,1]$. In dimension $d=2,$ this distribution has the density
\begin{equation} \label{circular2}
f(x,y) = \frac{1}{2 \pi \sqrt{1-x^2-y^2}} \bsone_{[x^2+y^2<1]}
\end{equation}
and in dimension $d=3$ this is given by the uniform distribution on the sphere of a unit ball. Notice that the bivariate case can be obtained as the two-dimensional marginals of the three-dimensional case.
Transforming the marginals to uniform distributions on $[0,1]$ via the transformation $X_i \mapsto (X_i+1)/2$ we get copulas called  \textit{spherical symmetric copulas}, which we will denote by
$C^\circ$. In Figure \ref{circ-copula} we show a discrete approximation of the bivariate spherical symmetric copula. Notice that the density is unbounded, going to infinity at the boundary of the support, and therefore in the discrete approximation there are many points there as one expects for a bivariate projection of points uniformly scattered on the sphere of a ball. 

\begin{figure}[!htbp]
\begin{center}
\includegraphics[width=0.6\textwidth,height=0.6\textwidth]{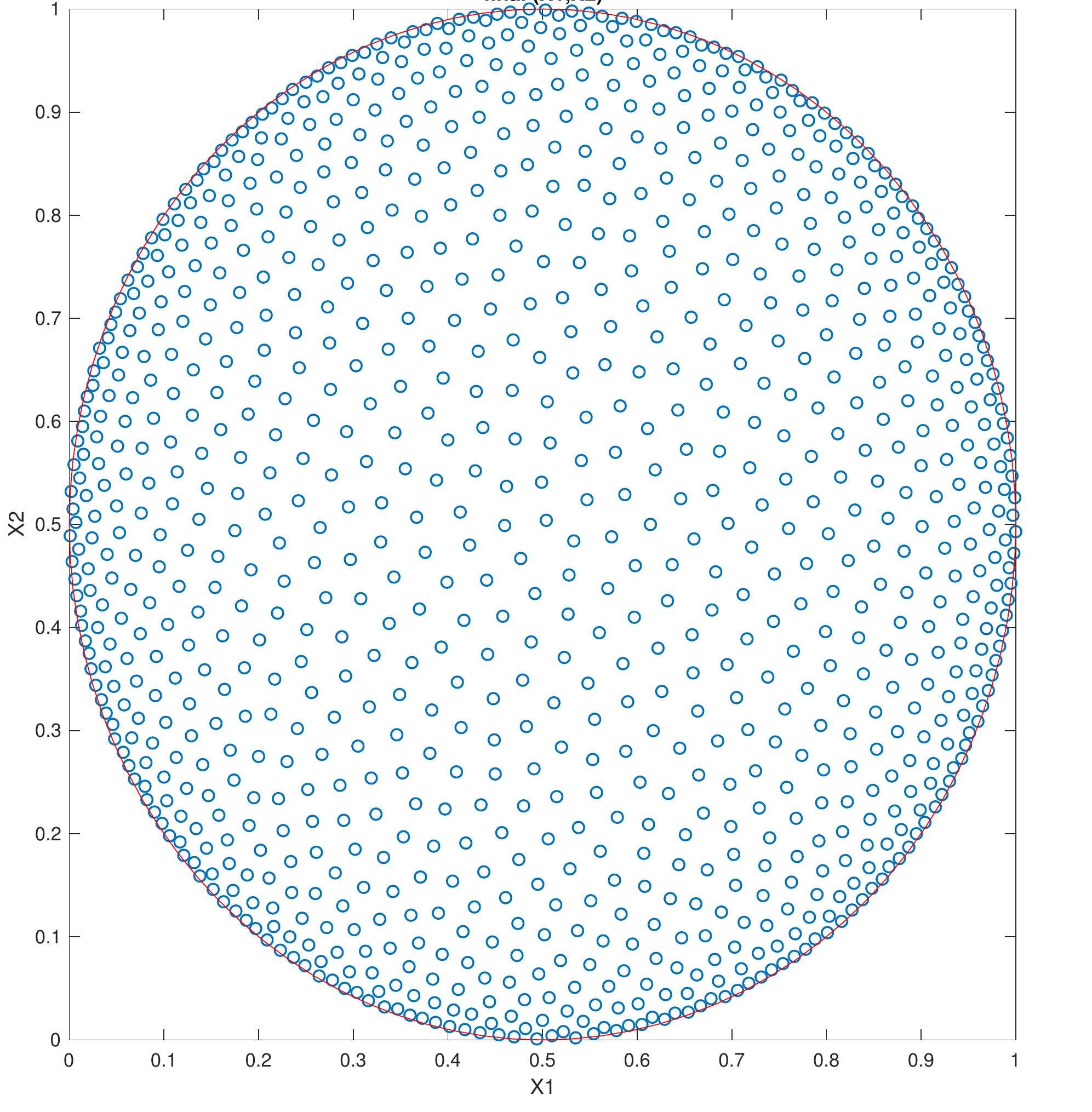}
\end{center}
\caption{An illustration of the bivariate spherical symmetric copula $C^\circ$ \label{circ-copula}}
\end{figure}

Moreover, we will use results that were obtained independently in \cite{mattner1993}, Theorem 2 and as main result in \cite{buja1994}. In our notation their results can be stated as follows.

\begin{theorem} \label{mattner}
The functional $S(F,F)$ is maximal among all random vectors $\bsX \sim F$ with $\E(\|\bsX\|^2) \le 1$ for $S(F^*,F^*)$, where $\bsX^* \sim F^*$ is given as follows.\\
If $d \ge 3$ then $\bsX^*$ is uniformly distributed on the sphere of a unit ball.\\
If $d=2$ then $\sqrt{2/3} \cdot \bsX^*$ has the density given in equation \eqref{circular2}.
\end{theorem}

From this result we can derive the following improved bounds for copulas.

\begin{theorem}\label{copula-sharp-upper-bound}
In dimension $d=2$ it holds for any copula $C$ that
$$
S(C,C) \le S(C^\circ, C^\circ) = \frac{\pi}{6}.
$$
In dimension $d=3$ it holds for any copula $C$ that
$$
S(C,C) \le S(C^\circ, C^\circ) = \frac{2}{3}.
$$
\end{theorem}

\begin{proof}
We define the shift $T(X_1,\ldots,X_d) = (X_1 - 1/2,\ldots,X_d - 1/2)$ that transforms marginals from uniform on $(0,1)$ to uniform on $(-1/2,1/2)$. As obviously 
$$
 \E\left( \| \bsX - \bsY \|_2 \right) =  \E\left( \| T\bsX - T \bsY \|_2 \right)
$$
this shift does not affect the functional $S$, and therefore we can replace the copulas by distributions $F$ with uniform marginals on $(-1/2,1/2)$.
For any such random vector $\bsX \sim F$ we get
$$
\E(\|\bsX\|^2) = \E(X_1^2 + \ldots + X_d^2) = \frac{d}{12} \mbox{ and hence } \E\left(\left\|\ \sqrt{\frac{12}{d}} \bsX\right\|^2\right) = 1.
$$
For any copula $C$, and $\bsX \sim C$ and  $\bsY \sim C$ independent, we thus obtain that $\E\left(\left\|\ \sqrt{\frac{12}{d}} T\bsX\right\|^2\right) = 1$ and thus from Theorem \ref{mattner},
\begin{align} \label{eq-circ}
S(C,C) & =  \E\left( \| \bsX - \bsY \|_2 \right) =  \E\left( \| T\bsX - T \bsY \|_2 \right) \\ \nonumber
& = \sqrt{\frac{d}{12}}  \E\left( \left\| \sqrt{\frac{12}{d}}T\bsX - \sqrt{\frac{12}{d}} T \bsY \right\|_2 \right) \le  \sqrt{\frac{d}{12}}  S(F^*,F^*),
\end{align}
with $S(F^*,F^*)$ as described in Theorem \ref{mattner}. In case $d=2$ and $d=3$ we get equality for $C = C^\circ$ and the corresponding values for 
$S(F^*,F^*)$ are computed in \cite{buja1994}, there denoted as $M_d$. They are given by 
$$
M_2 = \frac{\pi}{\sqrt{6}} \mbox{ and } M_3 = \frac{4}{3}.
$$
Thus we derive for $d=2$ that
$$
S(C,C) \le S(C^\circ, C^\circ) =  \sqrt{\frac{2}{12}} \cdot M_2 =  \frac{\pi}{6}
$$
and for $d= 3$ that
$$
S(C,C) \le S(C^\circ, C^\circ) = \sqrt{\frac{3}{12}}  \cdot M_3 = \frac{2}{3}.
$$
\end{proof}

We also get an improved bound for $d \ge 4$ from Theorem \ref{mattner} but then it is not sharp, as the uniform distribution on the sphere no longer has uniform marginals.
Indeed, \cite{perlman-wellner-11} show that there cannot exist any spherical symmetric copula in dimension $d \ge 4$ and therefore we do not know, how the copula $C$ that maximizes
$C \mapsto S(C,C)$ looks like. We conjecture that it in some sense will be close to spherical symmetry. The improved bound that we can derive from \eqref{eq-circ} and Theorem \ref{mattner} 
in case $d=4$ is
$$
S(C,C) \le \sqrt{\frac{4}{12}} \cdot  M_4 = \sqrt{\frac{1}{3}} \cdot  \frac{64}{15 \pi} \approx 0.784
$$
whereas the bound from Jensen's inequality for $d=4$ given in Corollary \ref{upper-bound-S-copula} is $\sqrt{2/3} \approx 0.816$. For higher dimensions $d$ the difference between the bounds derived from Jensen's inequality in Corollary \ref{upper-bound-S-copula}  and the better ones derived from \eqref{eq-circ} and Theorem \ref{mattner} become smaller and smaller as the latter ones are also approximately $\sqrt{d/6}$
for large $d$.

\begin{remark}
We also notice that $S(C^\circ, C^\circ)$ is not an upper bound for the case that we allow the copulas to be different, as we have an explicit counterexample in Example \ref{sc1c2}
with 
$$
S(C^+,C^\|) = \frac{1}{\sqrt{2}} \int_0^1 \int_0^1 \sqrt{\frac{1}{2} + \left(\frac{1}{2}+ x-y\right)^2} \ dx dy \approx 0.549 >   \frac{\pi}{6} = S(C^\circ, C^\circ).
$$
\end{remark}

\section{Bounds on the Energy Score \label{S3}}

Let $F$ be a distribution and $\bsy$ an observation, we now  first study bounds on $S_\beta(F,\bsy)$ in order to obtain bounds on the energy score $ES_{\beta}(F,\bsy)$.

\subsection{Bounds on $S_\beta(F,\bsy)$\label{coco}}

First, note that we cannot expect a general upper bound for $\bsy \mapsto S_\beta(F,\bsy)$ as $S_\beta(F,\bsy) \to \infty$ for $\bsy  \to \infty$. We thus concentrate on studying the lower bound in what follows.

%We now assume that we have an observation $\bsy$ and want to find bounds for the $S_\beta(F,\bsy)$. 
It is clear that as a function of $\bsy$,  the expression of $S_\beta(F,\bsy)$
is small if $\bsy$ is in some sense near the center of the distribution. In the univariate case it is a well-known simple result that $y \mapsto \E|X-y|$ is 
minimized for $y^* := F_X^{-1}(1/2)$ being the median of the distribution of $X$. 

We can prove a similar lower bound in the multivariate case for the upper Fr\'echet bound $F^+$, if we assume that the marginals are symmetric and unimodal. Recall that a univariate
distribution with cdf $F$ is called symmetric and unimodal with respect to some $\mu \in \R$, if $F(\mu + t) + F(\mu-t) = 1$ for all $t > 0$ and $F$ is convex on $(-\infty,\mu)$ and concave on $(\mu,\infty)$. 
We will need the following simple Lemma for such distributions that we state with a proof here, as we could not find it in the literature.

\begin{lemma} \label{lemma-unimodal}
Assume that the random variable $X$ has a continuous, unimodal and symmetric distribution with respect to $\mu$. Then it holds for all $\mu \le x < y$ and all $\mu \ge x > y$ that
$$
|X-x| \le_{st} |X-y|.
$$
\end{lemma}

\begin{proof}
Denote by $F_x$ the cdf of $|X-x|$, and assume  $\mu \le x < y$. We have to show that $F_x(t) \ge F_y(t)$ for all $t \ge 0$. A simple calculation shows
$F_x(t) = F(x+t) - F(x-t)$. As $F$ is continuous, unimodal and symmetric, it has a density $f$ which is symmetric around $\mu$ and decreasing on $[\mu,\infty)$. Therefore
$$
\frac{\partial}{\partial x} F_x(t) = f(x+t) - f(x-t) \le 0
$$
as $|x+t-\mu| > |x-t-\mu |$. This implies the assertion for $\mu \le x < y$ and the case $\mu \ge x > y$ follows then by symmetry.
\end{proof}

\begin{theorem} \label{bound-unimodal}
Assume that the random vector $\bsX$ has a cdf $F$ with marginals $F_i$ that are continuous, unimodal and symmetric with respect to $\mu_i$, $i = 1,\ldots,d$. 
Then we have for all $\bsy \in \R^d$
$$
S_\beta(F,\bsy) \ge S_\beta(F^+,\bsmu).
$$
\end{theorem}

\begin{proof}
We have 
\begin{align} \label{z-representation-y1}
S_\beta( F,\bsy) & =    \ \E\left( \sum_{i=1}^d Z_{i,y_i}^2  \right)^\frac{\beta}{2},
 \end{align}
 where $Z_{i,y_i}^2 = (X_i-y_i)^2$. It follows from Lemma \ref{lemma-unimodal} that
  $$
(X_i-y_i)^2 \ge_{st} (X_i-\mu_i)^2
 $$
 for all $i$. Let us denote by $(Z_{i,y_i}^+)^2, i=1,\ldots,d$ comonotone random variables with the same distributions as $Z_{i,y_i}^2$. 
Notice that for general $\bsy$ the vector $(Z_{1,y_1}^2, \ldots, Z_{d,y_d}^2)$ is typically not comonotone, even if $F$ is the comonotonic upper Fr\'echet bound $F=F^+$. This is the case, however, if $y_i = \mu_i$ is the median for all $i = 1,\ldots,d$. Therefore we get
 \begin{align} \label{z-representation-y2}
 S_\beta( F,\bsy) & =    \ \E\left( \sum_{i=1}^d Z_{i,y_i}^2  \right)^\frac{\beta}{2} \ \ge    \ \E\left( \sum_{i=1}^d  (Z_{i,y_i}^+)^2 \right)^\frac{\beta}{2}\\
 & \ge  \  \E\left( \sum_{i=1}^d  (Z_{i,\mu_i}^+)^2 \right)^\frac{\beta}{2} = S_{\beta}(F^+,\bsmu).
 \end{align}
The first inequality follows as in the proof of Theorem \ref{lower-bound-FG}.
The second inequality follows from the fact that for random vectors $\bsZ$ and $\bsZ'$ with the same copula $C^+$ stochastic ordering $Z_h \le_{st} Z_h'$ of the
marginals implies multivariate stochastic ordering and thus $\E f(\bsZ) \le \E f(\bsZ')$ for all increasing functions $f:\R^d \to \R$, see e.g., \cite{mueller2002}, Theorem 3.3.8. %M\"uller and Stoyan (2002)
\end{proof}

%\subsection{Lower Bounds on $S_\beta(C,\bsy)$ for copulas\label{coco}}

For copulas we get the following corollary. We denote here by $\bsone = (1,\ldots,1)$ a vector with all components being equal to one.

\begin{corollary} \label{bound-scy}
For all copulas $C$ and all $\bsy \in [0,1]^d$ it holds
$$
S_\beta(C,\bsy) \ge S_\beta\left(C^+,\frac{1}{2}\bsone\right) = d^{\frac{\beta}{2}} \frac{\frac{1}{2}^\beta}{\beta +1}.
$$
\end{corollary}

\begin{proof}
This follows from Theorem \ref{bound-unimodal},
as uniform distributions $U$ on $(0,1)$ are symmetric and unimodal with respect to $1/2$ and 
$$
\E\left(|U-1/2|^\beta\right) = \frac{\frac{1}{2}^\beta}{\beta +1}.
$$
\end{proof}

In the case of copulas we can also have a look at the case that the observation $\bsy$ is an extreme point, which can be assumed to be without loss of generality
$\bsy = \bszero$.

\begin{theorem}  \label{lower-bound-0}
For $\bsy = \bszero$ the function
$$
C \mapsto S_\beta(C,\bszero) 
$$
attains its minimum for the upper Fr\'echet bound  $C^+$. 
\end{theorem}

\begin{proof}
The proof for the minimum of $C \mapsto S_\beta(C,\bszero)$ follows the same lines as in Theorem \ref{lower-bound-FG}. If $\bsU$ is a random vector with copula $C$, then
$$
S_\beta(C,\bszero) = \E f(Z_1,\ldots,Z_d)
$$ 
where $Z_i = U_i^2, \ i=1,\ldots, d,$ and $f$ is a submodular function. As  $\bsZ = (Z_1,\ldots,Z_d)$ has the same copula as $\bsU$, we can conclude that $S_\beta(C,\bszero) \ge S_\beta(C^+,\bszero)$.
\end{proof}

It is easy to see that it is not true in general that the upper Fr\'echet bound  $C^+$ minimizes 
$C \mapsto S_\beta(C,\bsy)$. Due to invariance under rotations the minimum is obtained for the lower Fr\'echet bound  $C^-$ if $\bsy = (0,1)$.

\subsection{Bounds on $ES_\beta(F,\bsy)$}

Similarly to the above study of $S_\beta(F,\bsy)$, one cannot expect a general upper bound of the energy score 
$$ES_\beta(F,\bsy)=S_\beta(F,\bsy)-\frac{1}{2}S_\beta(F,F)$$
as the quantity tends to $+\infty$  for $\bsy  \to \infty$. As we have to deal with a difference of two quantities, it is also in general more difficult to find sharp bounds, whereas one easily gets
some bounds by bounding each of the two quantities using our previous results. 

We now first consider a bounded domain for $\bsy$. We then characterize the copula that achieves the lower bound.

As $\bsy \mapsto \|\bsx-\bsy \|_2^\beta$ is convex for $\beta \ge 1$, we also get that $\bsy \mapsto S_\beta(F,\bsy)$ is convex in this case. From Lemma \ref{LEM1}, we thus can easily derive the following result.

\begin{lemma} \label{lemma-convex}
For $\beta \in [1,2)$ the functions $\bsy \mapsto ES_\beta(F,\bsy)$ and $F \mapsto ES_\beta(F,\bsy)$ are convex.
\end{lemma}

Considering a copula $C$ and an observation $\bsy \in [0,1]^d$ we immediately get the following consequence.

\begin{proposition}
For $\beta \in [1,2)$ the function
$$
\bsy \mapsto ES_\beta(C,\bsy), \quad \bsy \in [0,1]^d,
$$
attains a maximum in the set $\{0,1\}^d$. 
\end{proposition}

\begin{proof}
This immediately follows from the fact that a convex function on a compact convex domain attains a maximum in an extreme point.
\end{proof}

Following the results on the lower bound of $S_\beta(F,\bsy)$ obtained for copulas in Section \ref{coco}, it is natural to conjecture that the upper Fr\'echet bound  $C^+$ is also the minimum of $C \mapsto ES_\beta(C,\bszero)$. This is not true, however. We can show that in the bivariate case $ES(C^+,\bszero) > ES(\hat{C},\bszero)$
for the copula $\hat{C}$ of the following random vector $\hat{\bsU}$. Let $U,U' \sim U(0,1)$ be independent uniformly distributed random variables and define $\hat{\bsU} = (\hat{U}_1,\hat{U}_2)$
as follows: $\hat{U}_1= \hat{U}_2 = U$, if $U \le 1/2$. If $U > 1/2$ then $\hat{U}_1=U$ and $\hat{U}_2=(U'+1)/2$.

We are not able to obtain an explicit lower bound but we can characterize some properties of the copula that achieves the minimum energy score.

Let $T: \R^d \to \R^d$ be a transformation that is an isometry and that preserves the marginal distributions in the following sense: if $\bsX$ is a random vector
with distribution function $F$ and marginals $F_1,\ldots,F_d$, then $T(\bsX)$ also has the same marginals and for any $\bsx,\bsy \in \R^d$ we have
$$
 \|T(\bsx)-T(\bsy) \|_2 =  \|\bsx-\bsy \|_2.
$$
It is easy to see that in the case of uniform marginals, i.e., for copulas, this holds for reflections of the form
$$
T_i(\bsx) = (x_1,\ldots,x_{i-1}, 1-x_i, x_{i+1},\ldots,x_d)
$$
and for permutations $\pi$ of the coordinates
$$
T_{\pi}(\bsx) = (x_{\pi_1},\ldots,x_{\pi_d}).
$$
Let us denote by $C_T$ the copula of the transformation $T(\bsU)$, if $\bsU$ is a random vector with copula $C$, and let $H$ be the finite group generated by all these isometric transformations of the hypercube that preserve the copula property. We get the following theorem.

\begin{theorem} \label{lower-bound-1/2}
For $\beta \in [1,2)$ and $\bsy = \frac{1}{2}\bsone$ the function 
$$
C \mapsto ES_\beta(C,\bsy)
$$
attains a minimum for a copula that is invariant under $H$. 
\end{theorem}

\begin{proof}
Let us define for a fixed $C$ and $\bsy = \frac{1}{2}\bsone$
$$
\hat{\bsy} := \frac{1}{ |H |} \sum_{T \in H} T(\bsy)
$$
and
$$
\hat{C} := \frac{1}{ |H |} \sum_{T \in H} C_T.
$$
Then $\hat{\bsy} = \bsy$, $T(\hat{\bsy}) = \hat{\bsy}$ for all $T \in H$ and $\hat{C} $ is invariant under $H$. Due to convexity of $C \mapsto ES_\beta(C,\bsy)$ we get
\begin{align*}
ES_\beta(\hat{C} ,\hat{\bsy})  & \le \frac{1}{ |H |} \sum_{T \in H} ES_\beta(C_T, \hat{\bsy}) = \frac{1}{ |H |} \sum_{T \in H} ES_\beta(C_T, T(\hat{\bsy})) \\
& = ES_\beta(C, \hat{\bsy}) .
\end{align*}
\end{proof}

With a similar argument we can show that the function $C \mapsto ES_\beta(C,\bszero)$ attains a minimum for a copula that is invariant under permutations, but we are not able to derive an explicit solution for the moment.

\section{Conclusions \label{S4}}

We investigated dependence uncertainty bounds on the energy score and for related functionals. The obtained results indicate that indeed these functionals seem not to be very sensitive to the dependence structure
as one can see e.g., from the inequality
$$
\frac{1}{3}\sqrt{d} \le S(C_1,C_2) \le  \sqrt{\frac{d}{6}}
$$
in Corollary \ref{upper-bound-S-copula}, which holds for all copulas $C_1,C_2$. Notice that we get the even closer sharp bounds
$$
\frac{\sqrt{2} }{3}\le S(C,C) \le \frac{\pi}{6}
$$
in Theorem \ref{copula-sharp-upper-bound} for the case $d=2$ if we restrict to the case of equal copulas. Therefore our results support the corresponding claim in \cite{pinson2013} which was based on a simulation study using multivariate normal distributions. However, many questions remain open and we hope to stimulate research on this topic that we consider as important. For example, we are not able to find explicitly the copula that achieves the lower bound of the energy score. We are only able to provide a partial characterization of it in Theorem \ref{lower-bound-1/2}. We are also working on the problem of finding the numerical solution of this optimization
problem by using a variant of the swapping algorithm that was used in \cite{puccetti2017} for a related problem. First results indicate that the solution seems to be a copula with a very unusual shape, but these results will be reported in a forthcoming paper.

\bibliographystyle{artbibst}
\bibliography{bounds}

\end{document}